\documentclass{amsart}
\usepackage[T2A]{fontenc}
\usepackage[cp1251]{inputenc}
\usepackage[english]{babel}
\usepackage[unicode]{hyperref}
\usepackage[left=3cm,right=2cm,top=2cm,bottom=2cm]{geometry}

\usepackage{amssymb}
\usepackage{amscd}
\newcommand\sly{\spaceskip .33333em plus 1em\sloppy}

\theoremstyle{plain}

\newtheorem{thm}{Theorem}[section]
\newtheorem{cor}[thm]{Corollary}
\newtheorem{lem}[thm]{Lemma}
\newtheorem{prop}[thm]{Proposition}
\newtheorem{problem}[thm]{Problem}
\numberwithin{equation}{section}

\theoremstyle{definition}

\newtheorem{rem}[thm]{Remark}

\newcommand\blank{\mathord{\hbox to 1.5ex{\hrulefill}}\,}

\let\opn\operatorname

\DeclareMathOperator\GG{G}
\DeclareMathOperator\Sp{Sp}
\DeclareMathOperator\St{St}

\DeclareMathOperator\E{E}
\DeclareMathOperator\Ker{Ker}
\DeclareMathOperator{\chr}{char}

\newcommand\Z{\mathbb Z}
\newcommand\LL{\mathcal L}
\newcommand\AAA{\mathfrak{A}}

\let\mf\mathfrak
\let\ph\varphi
\let\eps\varepsilon
\def\le{\leqslant}
\def\ge{\geqslant}
\let\ovl\overline
\let\wt\widetilde
\let\wh\widehat
\let\sm\smallsetminus
\def\No{\textnumero\,}
\def\llq{``}
\def\rrq{''}

\begin{document}

\title
[Subgroups of Chevalley groups of types $B_l$ and $C_l$]
{Subgroups of Chevalley groups of types $B_l$ and $C_l$, containing the group over a subring, and corresponding carpets}

\author
{Ya.~N.~Nuzhin}

\address{Institute of Mathematics\\ and Fundamental Informatics\\
Siberian Federal University\\
Svobodny prospect 79, Krasnoyarsk, 660041}

\email{nuzhin2008@rambler.ru}

\author
{A.~V.~Stepanov}

\address{St. Petersburg State University\\
Universitetskaya nab. 7-9,
St. Petersburg, 199034}

\email{stepanov239@gmail.com}

\keywords{classical groups, subgroup lattice, carpet subgroups, Bruhat decomposition}

\begin{abstract}
We continue the study of subgroups of the Chevalley group $G_P(\Phi,R)$ over a ring $R$ with root system $\Phi$ and
weight lattice $P$, containing the elementary subgroup $E_P(\Phi,K)$ over a subring $K$ of $R$.
A.~Bak and A.~V.~Stepanov considered recently the case of symplectic group (simply connected group with root system $\Phi=C_l$) 
in characteristic 2. In the current article we extend their result for the case $\Phi=B_l$ 
and for the groups with other weight lattices. As well as in the Ya.~N.~Nuzhin's work on the case where
$R$ is an algebraic extension of a nonperfect field $K$ and $\Phi$ is not simply laced
the description uses carpet subgroups parametrized by two additive subgroups.
In the second part of the article we establish the Bruhat decomposition for these carpet subgroups
and prove that they have a split saturated Tits system. As a corollary we obtain that they are
simple as abstract groups.
\end{abstract}

\thanks
{The work of the first author (\S\S5--9) is supported by RFBR (project \No16--01--00707).
The research of the second author (\S\S2--4) is supported bu the Russian Science Foundation (project \No17-11-01261).}

\maketitle


\section*{Introduction}
Let $\GG_P(\Phi,\,\blank\,)$ be a Chevalley--Demazure group scheme,
$\E_P(\Phi,\,\blank\,)$ its elementary subgroup, and let
$K\subseteq R$ be a pair of rings (by default all rings are commutative with a unit element 1 and all ring homomorphisms preserve 1)
In the current article we study the lattice $\LL$ of subgroups of $\GG_P(\Phi,R)$, containing $\E_P(\Phi,K)$, 
assuming that the root system~$\Phi$ is not simply laced, i.\,e. $\Phi=B_l$, $C_l$, $F_4$ or $G_2$.

For root simply laced root systems reasonable description can be obtained only if $R$ is quasi algebraic
over $K$, see.~\cite{StepGL,StepNonstandard}. The standard description for an algebraic field extension
was obtained in~\cite{Nuzhin}. If $R$ is the fraction field of a principal ideal domain $K$,
the same result was proved in~\cite{NuzhYak} (in this case $R$ is quasi algebraic over $K$ as well).
For not simply laced root systems the situation is quite different. In particular,
for doubly laced root systems ($\Phi=B_l$, $C_l$ or $F_4$) the standard description was obtained
in~\cite{StepStandard} for an arbitrary pair of rings provided that 2 is invertible in $K$.
To specify the meaning of the standard description of the lattice $\LL$ 
consider the lattice $L(D,G)$ consisting of subgroups of an abstract group $G$, containing 
a given subgroup $D$.
A subgroup 
$$
F\in L(D,G)
$$
is called \emph{$D$-full}, if the normal closure $D^F$ coincides with $F$.
A \emph{sandwich} of the lattice $L(D,G)$ is the set of all subgroups containing a given $D$-full subgroup $F$ 
and lying in its normalizer $N_G(F)$. We say that the lattice $L(D,G)$ satisfies the
\emph{sandwich classification theorem} if it splits into the union of sandwiches
(with our definition the union is disjoint as the only $D$-full subgroup in a sandwich equals
$D^H$ for all $H$ in a this sandwich).
Now, by the standard description of the intermediate subgroup lattice we mean the sandwich classification
theorem together with a description of all $D$-full subgroups.

In the cases of the standard description of the lattice $\LL$ mentioned above the sandwiches were parametrized
by subrings $S$ of $R$, containing  $K$, and $\E_P(\Phi,K)$-full subgroups were the elementary subgroups
$\E_P(\Phi,S)$. Observe that if $S$ and $R$ are fields, then the normalizer of
$\E_P(\Phi,S)$ in $\GG_P(\Phi,R)$ in equal to the product of $\GG_P(\Phi,S)$ by the center of
$\GG_P(\Phi,R)$. If the structure constants from the Chevalley commutator formula are not invertible,
then there exist other $\E_P(\Phi,K)$-full subgroups.

Denote by $p=p(\Phi)$ the maximal multiplicity of an edge in the Dynkin diagram or, what amounts to be the same, 
the maximal structure constant in the Chevalley commutator formula.
In other words, $p=2$ if $\Phi=B_l$, $C_l$ ($l\ge 2$), $F_4$, and $p=3$ if $\Phi=G_2$.
Let $\Lambda=(\Lambda_l,\Lambda_s)$ be a pair of additive subgroups of $R$, satisfying the following conditions:
\begin{enumerate}
\renewcommand\labelenumi{AP\arabic{enumi}.}
\item
$p\Lambda_s\subseteq \Lambda_l\subseteq \Lambda_s$;
\item
$t^p\Lambda_l\subseteq \Lambda_l$ для любого $t\in \Lambda_s$;
\item
$\Lambda_s$ is a subring if $\Phi\ne B_l$;
\item
$\Lambda_l$ is a subring if $\Phi\ne C_l$
\end{enumerate}
(if $\Phi=B_2=C_2$, then neither $\Lambda_s$, nor $\Lambda_l$ must be subrings).
If $\Phi=C_l$, $n\ge3$ such a pair is a particular case of a form ring in the sense of A.~Bak~\cite{BakBook}. 
In general we call such a pair an \emph{admissible pair of type $\Phi$}.
The notion of admissible pair is defined in the works E.~Abe and K.~Suzuki~\cite{AbeSuzuki,AbeNormal} in in a similar way.

Define a family of subgroups $\AAA=\{\AAA_\alpha \mid \alpha \in\Phi\}$ by
$$
\AAA_\alpha =
\begin{cases}
\Lambda_s,& \text{ if }\alpha  \text{ is a short root}, \\
\Lambda_l,& \text{ if }\alpha  \text{ is a long root.}
\end{cases}
$$
It turns out that this family is an elementary carpet of type~$\Phi$ in the sense of V.~M.~Levchuk~\cite{LevchukABA}. 
For the carpet $\AAA$, corresponding to an admissible pair $\Lambda$, its elementary carpet subgroup
$$
\E(\Phi,\Lambda)=\E(\Phi,\AAA)=\langle x_\alpha (\AAA_\alpha )\mid \alpha \in\Phi\rangle
$$
belongs to the lattice 
$\LL$, if $K\subseteq\Lambda_l$ and almost always is $\E(\Phi,K)$-full.
Thus, if the structure constants are not invertible, then there appear new sandwiches,
corresponding to admissible pairs instead of rings.

The sandwich classification theorem for the lattice $\LL$ with noninvertible structure constants
are established in the works~\cite{Nuzhin13} and~\cite{BakStepSymp}. The former considers the case of an algebraic field 
extension, whereas the latter deals with the group $\Sp_{2l}$ over arbitrary rings $K\subseteq R$
with $2=0$. It is well-known that in characteristic~2 Chevalley groups of types $B_l$ and $C_l$ are almost the same.
In particular, over perfect fields they are isomorphic. In the current article we extend the results of~\cite{BakStepSymp}
to all forms of Chevalley groups of types $B_l$ and $C_l$ and on the Steinberg group of type $C_l$. 
This is done by the homomorphisms $\GG(C_l,R)\leftrightarrows\GG(B_l,R)$ and group-theoretic arguments.
Conjecturally, for the groups of type $F_4$ the same result is also true, but the proof needs 
an extra knowledge about the normalizer of the elementary carpet subgroup. Therefore, it will be addressed in a
separate article.

Besides, for fields of bad characteristic we study elementary carpet subgroups corresponding to admissible
pairs. In particular, we prove the Bruhat decomposition for these subgroup and establish their simplicity.

The paper is organized as follows. In~\S\ref{sec:notation} we introduce the main notation to be used
in the article.
Section~\ref{sandwich} is devoted to the group-theoretic aspects of the sandwich classification theorem. 
In~\S\ref{sec:exceptional} we prove the existence of scheme morphisms
$$
\GG_{\mathrm{sc}}(C_l,\,\blank\,)\leftrightarrows\GG_{\mathrm{sc}}(B_l,\,\blank\,).
$$
Based on the results of two previous sections, in~\S\ref{sec:main} the sandwich classification theorem
is extended from simply connected group of type~$C_l$ to all\footnote{%
Since the center of the groups of type~$B_l$ and~$C_l$ in characteristic 2 is trivial, It may seem as they do not depend 
upon the weight lattice. However, the statement about the center holds only over reduced rings, and moreover,
the map from the simply connected group the adjoint one is not necessarily surjective even for reduced rings.}
all Chevalley groups of types~$B_l$ and~$C_l$ and to the Steinberg group of type~$C_l$.
In the remaining part of the article we study the carpet groups over fields. For the reader's convenience 
in~\S\ref{sec:carpets} we recall known statements about these groups to be used in the sequel.
In section~\ref{sec:Nuzhin} we state a corollary from the work~\cite{Nuzhin13} and pose some problems on
admissible pairs, which are negatively solved in~\S\ref{examples}.
Also in~\S\ref{examples} we discuss the question why there are no admissible pairs between
a principle ideal domain and its field of fractions.
In~\S\ref{sec:Bruhat} we establish the Bruhat decomposition in the carpet subgroup,
corresponding to an admissible pair.
The last section is to prove simplicity of this subgroup. This is done with help of the notion of 
$(B,N)$-pair.

\section{Main notation}\label{sec:notation}
Some definitions and notation were already formulated in the introduction, the others are given in
the current section.

Let $G$ be a group and $a,b,c\in G$. Denote by $a^b=b^{-1}ab$ the conjugate to
$a$ by $b$. The commutator $a^{-1}b^{-1}ab$ is denoted by $[a,b]$.
Let $F$ and $H$ be subsets of $G$. By $\langle F\rangle$ we denote the subgroup, generated by $F$. 
We denote by $F^H$ the subgroup of $G$, generated by the elements $a^b$ over all
$a\in F$ and $b\in H$. If
$H$ is a subgroup, $F^H$ is the smallest subgroup in $G$ containing $F$ and normalized by $H$.
The mutual commutator subgroup is a subgroup of~$G$ generated by all commutators
$[a,b]$, $a\in F$, $b\in H$. The normalizer of a subgroup $H$ in $G$ is denoted by~$N_G(H)$.

For a ring $R$ by $R^2$ we denote the set of all squares of elements of~$R$. In the current article
the set $R^2$ is considered for a ring $R$ of characteristic 2.
In this case $R^2$ is a subring.

Let $\Phi$ be a reduced irreducible root system and $\GG_P(\Phi,\,\blank\,)$ a Chevalley--Demazure group scheme
of type~$\Phi$ with the weight lattice~$P$. If the weight lattice is not important we write simply
$\GG(\Phi,\,\blank\,)$. The simply connected scheme (i.\,e. where $P=\opn{P}(\Phi)$\,) is denoted by
$\GG_{\mathrm{sc}}(\Phi,\,\blank\,)$.
Let $R$ be a ring. The elementary subgroup $\E(\Phi,R)$ of a Chevalley group $\GG(\Phi,R)$
is generated by the root subgroups
$$
X_\alpha(R)=\{x_\alpha (r)\mid r\in R\},
$$
over all $\alpha \in \Phi$. In case, where $R$ is a field or, more generally, a semilocal ring,
$\E_{\mathrm{sc}}(\Phi,R)$ coincides with the whole Chevalley group $\GG_{\mathrm{sc}}(\Phi,R)$.\footnote{%
In the book by R.~Steinberg~\cite{Steinberg} by a Chevalley group over a field the author means
its elementary subgroup. However, in non simply connected case this group is not algebraic.
For instance, the group $\opn{PSL}_n(F)$ is the elementary subgroup of the adjoint Chevalley group of type
$A_{n-1}$ over $F$, but cannot be defined by polynomial equations as soon as $F$ is infinite
and does not contain $n$th root of at least one element.}


For each $\alpha\in\Phi$ the scheme $X_\alpha$ is isomorphic to $\mathbb G_a$. The isomorphism $\mathbb G_a\to X_\alpha$
is denoted by $x_\alpha$. Thus for all $r,s\in R$ we have
\begin{equation}\label{add}
x_\alpha (r)x_\alpha (s)=x_\alpha(r+s).
\end{equation}
Furthermore, the Chevalley commutator formula
\begin{equation}\label{ChCommForm}
[x_\alpha(r),x_\beta(s)]=\prod_{i,j>0,\ i\alpha +j\beta\in\Phi} x_{i\alpha +j\beta}(C_{ij,\alpha \beta}r^is^j),\
\alpha\ne\pm\beta\in\Phi.
\end{equation}
holds in the elementary subgroup. Here $C_{ij,\alpha \beta}$ are nonzero integer constants not bigger than~$3$
(set $C_{ij,\alpha \beta}=0$ if $i\alpha +j\beta\notin\Phi$, this will allow to skip the condition $i\alpha +j\beta\in\Phi$).

The group with generators $y_\alpha(r)$, $\alpha\in\Phi$, $r\in R$ subject to relations~\ref{add} and~\ref{ChCommForm}
with all $x$ substituted by $y$ is called the Steinberg group of type~$\Phi$ over a ring~$R$. It is denoted by $\St(\Phi,R)$. 
The kernel of the natural epimorphism $\pi\colon \St(\Phi,R)\to\E(\Phi,R)$ that sends
$y_\alpha(r)$ to $x_\alpha(r)$ is denoted by $K_2(\Phi,R)$. 
Except certain root systems of small rank, the Steinberg group is centrally closed.
Conjecturally in all these cases $K_2(\Phi,R)$ is central in $\St(\Phi,R)$ and hence, 
is the Schur multiplier of the elementary subgroup.
Up to now the conjecture is proved for the root systems~$A_l$, $l\ge3$ (W.~van~der~Kallen~\cite{vdKCent}),
$C_l$, $l\ge3$ (A.~Lavrenov~\cite{LavSpCent}), $E_6,E_7,E_8$ (S.~Sinchuk~\cite{SinEl}) and
$D_l$, $l\ge4$ (A.~Lavrenov, S.~Sinchuk~\cite{LavSinDl}). 
We can not include the group $\St(B_l,R)$ in the theorem~\ref{main}
exactly because this conjecture is not proved for the group of type~$B_l$

Let us call a \emph{carpet of type $\Phi$ over $R$} a family of additive subgroups
$$
\AAA=\{\AAA_\alpha \mid \alpha \in \Phi\}
$$
of~$R$ satisfying the condition
\begin{equation}\label{carpet}
C_{ij,\alpha \beta}\AAA_\alpha ^i\AAA_\beta^j\subseteq\AAA_{i\alpha +j\beta},\ \text{при}\ \alpha\ne\pm\beta\in\Phi,\ i>0,\ j>0,
\end{equation}
where
$
\AAA_\alpha ^i=\{a^i\mid a\in\AAA_\alpha \},
$
and the constants $C_{ij,\alpha \beta}$ are defined by the Chevalley commutator formula.

An elementary carpet $\AAA$ of type $\Phi$ over $R$ defines the elementary \emph{carpet} subgroup
$$
\E(\Phi,\AAA)=\langle x_\alpha (\AAA_\alpha )\mid \alpha \in\Phi\rangle
$$
of $\GG(\Phi,R)$. In the current article we deal only with elementary carpets and elementary carpet
subgroups. Therefore, we shall skip the word ``elementary'' in this context.
A carpet $\AAA$ of type $\Phi$ over a ring $R$ is called \emph{closed}, if its carpet subgroup does not contain
extra root elements, i.\,e. if
$$
\E(\Phi,\AAA)\cap x_\alpha (R)=x_\alpha (\AAA_\alpha )\text{ for all }\alpha \in \Phi.
$$

In the scheme $\GG(\Phi,\,\blank\,)$ we fix a split maximal torus $T$ and an ordering of the root system.
Denote by $N$ the scheme normalizer of the torus~$T$ and let $U$ be the unipotent radical of the Borel subgroup:
$$
U(R)=\langle x_\alpha (R)\mid \alpha \in\Phi^+\rangle.
$$
In an appropriate matrix representation we may assume that $T(R)$ is the set of diagonal matrices,
$N(R)$ is the set of monomial matrices, and $U(R)$ is the set of upper unitriangular matrices inside $\GG(\Phi,R)$.

For a ring $R$ with a connected spectrum the quotient group $N(R)/T(R)$ is isomorphic to the Weyl group $W=W(\Phi)$
of $\Phi$. Note that over any ring~$R$ there exists a preimage of a given element $w\in W$ in $N(R)$,
because it exists already in $N(\Z)$.

\section{Sandwich classification theorem}\label{sandwich}
In this section we develop group theoretic methods of proof of the sandwich classification theorem
First, recall two facts obtained in the works~\cite{StepPolynormal,StepNormal,StepStandard}, see also~\cite{StepGrTh}. 
In the sequel the group is called perfect, if it coincides with its commutator subgroup.

Let $D$ be a perfect subgroup of an abstract group~$G$. 
Then for a subgroup $H\in L(D,G)$ the normal closure $D^H$ coincides with the mutual commutator subgroup $[H,D]$.
It is also clear that the sandwich classification theorem for the lattice $L(D,G)$ is equivalent to
the equalities $[H,D,D]=[H,D]$ for all subgroup $H$ from this lattice.

\begin{lem}[{\cite[Lemma~1]{StepPolynormal}}]\label{strong}
Let $D$ be a perfect subgroup of a group~$G$, normalizing a subgroup $H\le G$.
Then $[H,D,D]$ is normal in $H$ if and only if $[H,D,D]=[H,D]$.
\end{lem}

\begin{lem}[{\cite[Proposition~1.9]{StepNormal}}]\label{surj}
Let $D\!\le\! G$ and let $\ph\colon G\!\to\!\ovl G$ be a group epimorphism.
The sandwich classification theorem for the lattice $L(D,G)$ implies the sandwich classification theorem for the lattice
$L(\ph(D),\ovl G)$. Moreover, $\ph(D)$\nobreakdash-full subgroups of $\ovl G$ are the images of
$D$-full subgroups of~$G$ and their normalizers are the images of the normalizers of the corresponding
$D$\nobreakdash-full subgroups of $G$.
\end{lem}

The next two lemmas allows to lift the sandwich classification theorem to a central extension.

\begin{lem}\label{central0}
Let $D$ be a perfect subgroup of $G$, \ $F$ a $D$-full subgroup of $G$, and 
$\pi\colon S\twoheadrightarrow G$ an epimorphism with a central kernel.
\begin{enumerate}
\item
There exists the smallest subgroup $\wt D\le S$ such that $\pi(\wt D)=D$.
\item
Let $\wt F$ be the smallest, whereas $\wh F$ an arbitrary preimage of $F$ under $\pi$
(since any $D$-full subgroup is perfect, $\wt F$ exist in accordance to the first item).
For subgroups $\wh D$ and $\wh F$ the group $S$ such that $\pi(\wh D)=D$ and $\pi(\wh F)=F$
we have $[\wh F,\wh F]=[\wh D,\wh F]=\wt F$. In particular the group $\wt D$ is perfect
and $\wt F$ is $\wt D$-full.
\end{enumerate}
\end{lem}

\begin{proof}
Put $\wt D\!=\![\pi^{-1}(D),\pi^{-1}(D)]$ и $\wt F\!=\![\pi^{-1}(F),\pi^{-1}(D)]$. Since $D$ is perfect
and $F$ is $D$-full, $\pi(\wt D)\!=\!D$ and ${\pi(\wt F)\!=\!F}$.
For arbitrary preimages $\wh D$ and $\wh F$ of the groups $D$ and $F$ respectively we have
$\pi^{-1}(D)=\wh D\Ker\ph$ and $\pi^{-1}(F)=\wh F\Ker\ph$.
Hence,
$$
\wt F=[\pi^{-1}(F),\pi^{-1}(D)]=[\wh F\Ker\ph,\wh D\Ker\ph]=[\wh F,\wh D]
$$
as $\Ker\ph$ lies in the center of $S$. For $F=D$ we get $\wt D\le\wh D$,
otherwise, for $\wh D=\wt D$ we conclude that $\wt D$ is perfect. In general it is clear that $\wh F$ contains
$\wt D$. This observation with $\wh D=\wt D$ implies that
$\wt F=[\wh F,\wt D]\le \wh F$. Thus,
$\wt F$ as well as $\wt D$ is the smallest preimage.
Finally, the second assertion of the lemma is the general case of the displayed formula.

Note that $\wt F$ can be defined by the same formula as $\wt D$, i.\,e.
$\wt F=[\wh F,\wh F]$.
\end{proof}

\begin{lem}\label{central}
In the notation of the previous lemma the sandwich classification theorems for the lattices $L(D,G)$ and $L(\wt D,S)$
are equivalent and $\pi$ induces a bijection between the set of all $\wt D$-full subgroups of $S$ onto the set of
of all $D$-full subgroups of $G$.
\end{lem}

\begin{proof}
It follows from the Lemma~\ref{surj} that the sandwich classification theorem for the lattice $L(\wt D,S)$ implies
the sandwich classification theorem for the lattice $L(D,G)$ and the map induced by $\pi$ is surjective.
Injectivity of this map follows immediately from the previous lemma as in the set of all preimages
of a $D$-full subgroup is only one $\wt D$-full, namely the smallest one.

Now, suppose that the sandwich classification theorem holds for the lattice $L(D,G)$.
Let $H\in L(\wt D,S)$. Then $\pi(H)$ normalizes some $D$-full subgroup~$F$. 
Let $\wt F$ be the smallest preimage of~$F$ in~$S$. Then
$\pi([H,\wt D])=[\pi(H),D]=F$. By the previous lemma 
$\bigl[H,\wt D,\wt D\bigr]=\bigl[[H,\wt D],[H,\wt D]\bigr]=\wt F$.
Hence, $[H,\wt D,\wt D]$ is normal in $H$ as a commutator subgroup of a normal subgroup.
Now, Lemma~\ref{strong} implies that $[H,\wt D]=[H,\wt D,\wt D]=\wt F$, which completes the proof.
\end{proof}

Clearly, if $G\ge E\ge D$, then the sandwich classification of the lattice $L(D,G)$ implies
the sandwich classification of the lattice $L(D,E)$. Moreover, $D$-full subgroups of the latter lattice are just
$D$-full subgroups of the lattice $L(D,G)$ contained in $E$.
The converse is not true in general, even if we assume that $D$ is perfect, $H$ normal in $G$,
and $G/H$ is abelian.%
\footnote{Let $S$ be a nonabelian simple group, $E=S*S$, $D$ the normal closure of one of free factors in $E$, and
$G\cong E\rtimes \Z/2\Z$ extension of $E$ by the automorphism that changes free summands. Then the lattice $L(D,E)$ 
consists of one sandwich,
$D^G=E$, whereas $D^{D^G}=D$, i.\,e. the sandwich classification theorem does not hold for $L(D,G)$.}

However, under certain additional assumption this implication holds.
This statement will allow us to extend the sandwich classification theorem from the elementary subgroup
to the whole Chevalley group.

\begin{lem}\label{extendSC}
Let $D\le E\triangleleft G$. Assume that the sandwich classification theorem holds for $L(D,E)$.
Suppose further that the following conditions are satisfied.
\begin{enumerate}
\item\label{generated}
The group $D$ is generated by a union of finitely generated perfect subgroups.
\item\label{quasi-soluble}
Given a $D$-full subgroup $F\le E$, the quotient group $N_E(F)/F$ is quasi-solvable, i.\,e. is the union of
an ascending chain of solvable groups.
\end{enumerate}
Then the lattice $L(D,G)$ satisfies the sandwich classification theorem and the sets of 
$D$-full subgroups in $G$ and $E$ coincide.
\end{lem}

\begin{proof}
Since $D$ is generated by a union of perfect subgroups, it is perfect itself.
Then any $D$-full subgroup $F$ is perfect as well as it is generated by perfect subgroups $D^f$ 
over all $f\in F$.

Let $H\le G$. Since $E$ is normal in $G$, the subgroup $D^H$ belongs to the lattice $L(D,E)$ and hence,
is contained in some sandwich $L(F,N)$. Let $\wt D$ be a perfect finitely generated subgroup
in $D$ and $h\in H$. Then the subgroup $\wt D^h$ is finitely generated and perfect.
By condition~\eqref{quasi-soluble} $N$ is the union of an ascending chain of subgroups
$N_i$ such that $N_i/F$ are solvable. In particular, $F$ is the largest perfect subgroup in each $N_i$.
Sine $\wt D^h$ is finitely generated it is contained in
$N_i$ for some index $i$. Hence, $\wt D^h\le F$. Since $h$ is an arbitrary element of the group $H$ and $D$ is generated by
its perfect finitely generated subgroups, we obtain $D^H\le F$. Thus, given a subgroup~$H$ in~$G$ the group $D^H=F$ is 
$D$-full, as required.

The statement about the sets of $D$-full subgroups is obvious.
\end{proof}

Properties~(1) and~(2) of the previous lemma are internal properties of the group~$E$.
With a help of Lemma~\ref{surj} it easy to see that these properties are preserved by epimorphisms.
Therefore, the sandwich classification theorem can be extended from an epimorphic image of the group~$E$
to its arbitrary extension.

\begin{cor}\label{cor:extendSC}
Let $D\le E'\le E$.
Suppose that groups~$D$ and~$E$ satisfy the conditions of Lemma~{\rm\ref{extendSC}},
the sandwich classification theorem holds for the lattice $L(D,E)$,
and let $\ph\colon E'\to\ovl E$ be an epimorphism. Then, given a subgroup $\ovl G\triangleright\ovl E$
the lattice $L\bigl(\ph(D),\ovl G\bigr)$ satisfies the sandwich classification theorem and
$\ph(D)$-full subgroups of this lattice are the images of $D$-full subgroups of the group $E'$.
\end{cor}

\begin{proof}
Clearly the lattice $L(D,E')$ satisfies the sandwich classification theorem, $D$-full subgroups of this lattice
are $D$-full subgroups in $E$, and $N_{E'}(F)=N_E(F)\cap E'$.
Therefore, the conditions of the previous lemma hold for groups $D$ and $E'$ as well.

By Lemma~\ref{surj} the lattice $L\bigl(\ph(D),\ovl E\bigr)$ satisfies the sandwich classification theorem,
$\ph(D)$-full subgroups of $\ovl E$ are the images of 
$D$-full subgroups of $E'$, and their normalizers are the images of the normalizers of the corresponding $D$-full subgroups
of $E'$. Obviously, the property~\eqref{generated} of Lemma~\ref{extendSC} as well as the property of being
quasi-solvable are inherited by epimorphic images. 
Now the result follow from Lemma~\ref{extendSC}.
\end{proof}

\section{Exceptional morphism}\label{sec:exceptional}
In this section for group schemes over $\mathbb F_2$ we construct morphisms from
$\Sp_{2l}=\GG_{\mathrm{sc}}(C_l,\,\blank\,)$ to $\opn{Spin}_{2l+1}=\GG_{\mathrm{sc}}(B_l,\,\blank\,)$ and back
whose composition in any order is equal to the Frobenius endomorphism.
On elementary groups over fields these morphisms are mensioned in the book by R.~Steinberg~\cite[теорема~28]{Steinberg}.
For $l=2$ over a finite field this morphism is a key point in a construction of the Suzuki groups.

We start with recalling the construction of the Steinberg group.
Let $R$ be an algebra over the field~$\mathbb F_2$ and $\Phi\ne G_2$.
In this case the Steinberg group $\St(\Phi,R)$ is generated by symbols $y_\alpha(r)$, where
$\alpha\in\Phi$, and $r\in R$ subject to relations:
{\allowdisplaybreaks\begin{enumerate}
\item
$y_\alpha(r)y_\alpha(s)=y_\alpha(r+s)$;
\item
$[y_\alpha(r),y_\beta(s)]=1$, if $\alpha+\beta\notin\Phi$ or $\alpha\perp\beta$;
\item
$[y_\alpha(r),y_\beta(s)]=y_{\alpha+\beta}(rs)$, if $\alpha$, $\beta$, and $\alpha+\beta$ are of the same length;
\item
$[y_\alpha(r),y_\beta(s)]=y_{\alpha+\beta}(rs)y_{\alpha+2\beta}(rs^2)$, if
$\alpha$ is long, whereas $\beta$ and $\alpha+\beta$ are short roots.
\end{enumerate}}

To avoid a confusion let us denote the generators of the Steinberg group of type~$C_l$ by $c_\alpha(r)$
and of the Steinberg group of type~$B_l$ by $b_\alpha(r)$.
We use the standard presentation of the root systems~$B_l$ and~$C_l$ in a euclidean space with an orthonormal basis~$e_i$:
$$
B_l=\{e_i,\ e_i-e_j\mid 1\le i,j\le l,\ i\ne j\} \text{ and }
C_l=\{2e_i,\ e_i-e_j\mid 1\le i,j\le l,\ i\ne j\}.
$$

Define maps $\ph$ and $\psi$ between the generating sets of the groups $\St(C_l,R)$ and $\St(B_l,R)$ by the
following formulas.
%
%
%
\begin{equation}\label{DefMorphisms}
\begin{aligned}
\ph\bigl(c_\alpha(r)\bigr)=
\begin{cases}b_{\alpha/2} (r),&\text{if $\alpha$ is long},\\
             b_{\alpha} (r^2),&\text{if $\alpha$ is short};\end{cases}
\\
\psi\bigl(b_\alpha(r)\bigr)=
\begin{cases}c_{\alpha}   (r),&\text{if $\alpha$ is long},\\
             c_{2\alpha}(r^2),&\text{if $\alpha$ is short}.\end{cases}
\end{aligned}
\end{equation}
(the same letters will denote the maps between the generating sets of the groups $\E(C_l,R)$ and $\E(B_l,R)$\,).
It is easy to verify that both~$\ph$ and~$\psi$ take relations to relations. Hence, they can be extended
to group homomorphisms
$
\St(C_l,R)\mathop{\leftrightarrows}\limits_\ph^\psi\St(B_l,R).
$

Recall that the group $K_2(\Phi,R)$ is the kernel of the natural map $\St(\Phi,R)\to\GG_{\mathrm{sc}}(\Phi,R)$
sending $y_\alpha(r)$ to $x_\alpha(r)$.
If $F$ is a field, then $K_2(\Phi,F)$ is generated by the elements $\{r,s\}=h_\alpha(r)h_\alpha(s)h_\alpha(rs)^{-1}$
over all $\alpha\in\Phi$ and invertible elements $r,s\in F$,
where $h_\alpha(r)=w_\alpha(r)w_\alpha(-1)$ and $w_\alpha(r)=y_\alpha(r)y_{-\alpha}(-r^{-1})y_\alpha(r)$,
see e.\,g.~\cite[\S6]{Steinberg}.
In case $\Phi\ne C_l$ the elements $\{r,s\}$ are the Steinberg symbols, i.\,e. they satisfy
certain relations. If $\Phi=C_l$ there are less relations between these elements, but this is not important for our purposes.
Straightforward computation shows that the generators of $K_2$ goes to $K_2$ under the homomorphisms $\ph$ and $\psi$, 
therefore these maps induce the group homomorphisms between $\GG_{\mathrm{sc}}(C_l,F)=\E_{\mathrm{sc}}(C_l,F)$ and
$\GG_{\mathrm{sc}}(B_l,F)=\E_{\mathrm{sc}}(B_l,F)$. Denote these homomorphisms by~$\bar\ph$ and~$\bar\psi$ respectively.
It is easy to see that the images of these homomorphisms are equal to the elementary groups, corresponding to
the admissible pair $(F,F^2)$.

Our next aim is to show that these homomorphisms are regular, i.\,e. that they are induced by group scheme morphisms,
and hence, are defined over an arbitrary ring. Morally, this follows from the fact that a simply connected Chevalley
group is the sheafification of its elementary subgroup, but formally we can not apply this argument as
$\bar\ph$ and $\bar\psi$ even are not defined over all local rings.
The leading idea of the proof is that the morphism of affine schemes is uniquely defined by the image of
the generic element and that the affine algebra of a Chevalley group is a domain.

The notion of generic element rarely appears in the theory of algebraic groups, therefore we recall
some relevant definitions.
Let $G$ be an arbitrary affine scheme over a ring~$K$. By definition, $G$ is a functor from the category of $K$-algebras
to the category of sets, isomorphic to the functor
$\opn{Hom}_{K-\mathrm{alg}}(A,\,\blank\,)$, where $A=K[G]$ is the affine algebra of~$G$.
Thus, for any $K$-algebra $R$ the element
$h\in G(R)$ corresponds to the $K$-algebra homomorphism $A\to R$, which will be denoted by~$\eps_h$. 
The generic element $g_G$ of the affine scheme~$G$ is an element of $G(A)$, corresponding to the identity
homomorphism $\mathrm{id}\colon A\to A$.
If $G'$ is another scheme over $K$, then the scheme morphism $\theta\colon G\to G'$ is uniquely defined
by the image of the generic element~$g_G$ in the set~$G'(A)$ or, what amounts to be the same, 
the $K$-algebra homomorphism $K[G']\to A$. In more detailes this point of view to affine schemes is
discussed in the work by Demazure--Gabriel~\cite{DemazureGabriel}, see also the book~\cite{Jantzen}.

\begin{lem}
Let $G$ and $G'$ affine group schemes over a domain~$K$ and let $E\le G$ be a group subfunctor.
Suppose that $G$ is smooth and connected and that $G(R)=E(R)$ for all local rings~$R$.
Let $\theta\colon E\to G'$ be a natural transformation of the restrictions of the functors~$E$ and~$G'$ on
the full subcategory of $K$-algebras that are domains.
Then there exists a unique scheme morphism $\widetilde\theta\colon G\to G'$ such that for any domain~$R$
the restriction of $\widetilde\theta_R\colon G(R)\to G'(R)$ to $E(R)$ coincides with $\theta_R$.
\end{lem}

\begin{proof}
Let $A=K[G]$ be the affine algebra of the scheme~$G$. Since $G$ is smooth and connected and $K$ is a domain,
$A$ is a domain as well. Denote by $F$ its field of fractions.
Consider the restriction of the homomorphism $\theta_F\colon E(F)=G(F)\to G'(F)$ to the group $G(A)$.
We claim that the image of this restriction lies in~$G'(A)$.
Let $\mf p$ be a prime ideal of~$A$. Then
$$
\theta_F\bigl(G(A)\bigr)\le\theta_F\bigl(G(A_{\mf p})\bigr)=\theta_F\bigl(E(A_{\mf p})\bigr)=
\theta_{A_{\mf p}}\bigl(E(A_{\mf p})\bigr)\le G'(A_{\mf p}).
$$
Since $\bigcap\limits_{\mf p\in\opn{Spec}_A}A_{\mf p}= A$, the image of $G(A)$ under $\theta_F$
is contained in $G'(A)$. In particular, $g'=\theta_F(g_G)\in G'(A)$, where $g_G\in G(A)$ is the generic
element of the scheme~$G$. Now, for a $K$-algebra $R$ and an element $h\in G(R)$
put $$\widetilde\theta_R(h)=G'(\eps_h)(g').$$

As we have mentioned above, the element $g'$ defines the scheme morphism $G\to G'$ uniquely,
which immediately implies the uniqueness statement. It remains to prove that for any domain~$R$
the restriction $\widetilde\theta_R\colon G(R)\to G'(R)$ to $E(R)$ coincide with $\theta_R$.
Clearly, it suffices to give a proof for a field~$R$ (the fraction field of a domain).
Let $h\in G(R)$. The kernel of the homomorphism $\eps_h\colon A\to R$ is a prime ideal. Denote it by~$\mf p$.
Let $\eps$ be the homomorphism $A_{\mf p}\to R$ induced by $\eps_h$. Consider the diagram
\begin{equation*}
\begin{split}
\begin{CD}
\llap{$G(A_{\mf p})=\,$}E(A_{\mf p})@>{\theta_{A_{\mf p}}}>> G'(A_{\mf p})\\
@V{E(\eps)}VV                                               @VV{G'(\eps)}V\\
\llap{$G(R)=\,$}E(R)                @>{\theta_R}>>           G'(R)
\end{CD}
\end{split}
\end{equation*}
Since $\theta$ is a natural transformation, this diagram is commutative.
By definition of the homomorphism $\eps$ the image of the generic element $g_G$ under $\eps$ equals~$h$ 
(we identify elements of the groups $G(A)$ and $G'(A)$ with their canonical images in
$G(A_{\mf p})$ and $G'(A_{\mf p})$ respectively).

Therefore, the image of $g_G$ in $G'(R)$ equals $\theta_R(h)$. Choosing another path we see that
$g_G$ goes to $G'(\eps)(g')=G'(\eps_h)(g')=\widetilde\theta_R(h)$. Thus,
$\widetilde\theta_R(h)=\theta_R(h)$, which completes the proof.
\end{proof}

\begin{cor}\label{exceptional}
There exist morphisms of group schemes
$$
\begin{CD}
\GG_{\mathrm{sc}}(C_l,\,\blank\,)@>\widehat\ph >>\GG_{\mathrm{sc}}(B_l,\,\blank\,)@>\widehat\psi >>\GG_{\mathrm{sc}}(C_l,\,\blank\,)
\end{CD}
$$
over the field $\mathbb F_2$ defined on the root unipotent elements by formulas~\eqref{DefMorphisms}.
The composition of these morphisms in any order is equal to the Frobenius endomorphism.
\end{cor}

When this text has been already written A.~V.~Smolenski in his work~\cite{SmolSuzRee}
has given explicit formulas for these morphisms.

Over a perfect field these homomorphisms are isomorphisms and over a reduced rings they are injective.
However, as group scheme morphisms they are epimorphisms but not monomorphisms. Indeed, the kernel of
both of them is a direct product of several copies of the scheme~$\mu_2$, whereas the image is dense.

\section{Distribution of subgroups}\label{sec:main}
In this section we prove the sandwich classification theorem for the lattices
$L\bigl(\St(\Phi,K),\St(\Phi,R)\bigr)$ and $L\bigl(\E(\Phi,K),\GG(\Phi,R)\bigr)$,
where $K\subseteq R$ are $\mathbb F_2$-algebras, $\Phi=B_l$ or $C_l$, and $\GG$ is not necessarily
simply connected.

\begin{thm}\label{main}
Let $K\subseteq R$ be $\mathbb F_2$-algebras, $\Phi=B_l$ or $C_l$, and $l\ge3$.
Then given a subgroup $H$ of $\GG(\Phi,R)$, containing $\E(\Phi,K)$,
there exists a unique admissible pair~$\Lambda$ if type~$\Phi$ such that
$$
\E(\Phi,\Lambda)\le H\le N_{\GG(\Phi,R)}\bigl(\E(\Phi,\Lambda)\bigr).
$$

Similarly, given a subgroup $H\le\St(C_l,R)$, containing $\St(C_l,K)$,
there exists a uique admissible pair $\Lambda$ such that
$$
\St(C_l,\Lambda)\le H\le N_{\St(C_l,R)}\bigl(\St(C_l,\Lambda)\bigr).
$$
\end{thm}

\begin{proof}
The statement of the theorem for an arbitrary subring~$K$ of~$R$ follows from the statement for $K=\mathbb F_2$.
Therefore, without loss of generality we may assume that $K=\mathbb F_2$.

In case of the simply connected Chevalley group $\Sp_{2l}$ with the root system~$C_l$ the result is obtained in the 
paper~\cite{BakStepSymp}.
Clearly, it implies the stadard description of the lattice 
$L\bigl(\E_{\mathrm{sc}}(C_l,K),\E_{\mathrm{sc}}(C_l,R)\bigr)$.
By Lavrenov's theorem~\cite{LavSpCent} the kernel of the canonical homomorphism
$\St(C_l,R)\to\E_{\mathrm{sc}}(C_l,R)$ lies in the center. By Corollary~4.4 from~\cite{SteinChevalley}
the groups $\St(C_l,K)$ and $\E_{\mathrm{sc}}(C_l,K)$ are perfect. Therefore, Lemma~\ref{central}
implies the second assertion.

Let $S$ be an $R$-algebra generated by symbols $t_r$ over all $r\!\in\! R$ subject to relations $t_r^2\!=\!r$.
Since $2\!=\!0$ in $R$, squaring is a ring homomorphism, hence, $S^2\!\!=\!R$, i.\,e.
$\Lambda\!=\!(S,R)$ is an admissible pair in $S$ of type $C_l$.
The standard description of the lattice $L\bigl(\E_{\mathrm{sc}}(C_l,K),\E_{\mathrm{sc}}(C_l,S)\bigr)$ implies
the standard description of the lattice $L\bigl(\E_{\mathrm{sc}}(C_l,K),\E_{\mathrm{sc}}(C_l,\Lambda)\bigr)$.
The function $\ph$ maps the generators of the group $\E_{\mathrm{sc}}(C_l,\Lambda)$ onto the generators
of the group $\E_{\mathrm{sc}}(B_l,R)$. Therefore the function~$\widehat\ph_S$ from Corollary~\ref{exceptional}
maps $\E_{\mathrm{sc}}(C_l,\Lambda)$ onto $\E_{\mathrm{sc}}(B_l,R)$. Since $K^2=K$, the image of the group
$\E_{\mathrm{sc}}(C_l,K)$ equals $\E_{\mathrm{sc}}(B_l,K)$.
For any ring~$R$, any root system~$\Phi$ and any weight lattice~$P$,
the canonical map from $\E_{\mathrm{sc}}(\Phi,R)$ to $\E_P(\Phi,R)$ is surjective.
Therefore, there exists an epimorphism onto $\E_P(C_l,R)$ or $\E_P(B_l,R)$ from a certain
subgroup of the group $\opn{ESp}_{2l}(Q)=\E_{\mathrm{sc}}(C_l,Q)$, containing
$\opn{ESp}_{2l}(K)=\E_{\mathrm{sc}}(C_l,K)$ (where $Q=R$ or $Q=S$).

By G.~Taddei's theorem~\cite{Taddei} $\E_P(\Phi,R)$ is normal in $\GG_P(\Phi,R)$,
Hence, by corollary~\ref{cor:extendSC} for the proof of the theorem it suffices to verify conditions~(1) and~(2) 
of Lemma~\ref{extendSC} for the groups $\opn{ESp}_{2l}(K)\le\opn{ESp}_{2l}(Q)$.
The group $\opn{ESp}_{2l}(K)$ is finite and perfect as we assume that $K=\mathbb F_2$. 
On the other hand, the property~\ref{quasi-soluble} follows immediately from Theorem~2 of~\cite{BakStepSymp}, 
which completes the proof.
\end{proof}

\section{Preliminary results on carpet subgroups over a field}\label{sec:carpets}
In this section we present known statements on carpet subgroups of a field that we shall use in a sequel.
The following lemma appears first in~\cite{Suzuki}. It is a particular case of Theorem~3 from~\cite{LevchukABA}. 
In a sequel by default $F$ denotes an arbitrary field and all computations are performed in a Chevalley group
$\GG(\Phi,F)$.

\begin{lem}\label{lem1}
Suppose that a subgroup $M\le U(F)$ is normalized by $T(K)$ for a subfield $K\subseteq F$ such that $|K|>4$.
If
$
x_{\alpha _1}(t_1)\dots x_{\alpha _k}(t_k)\in M,
$
where
$
\alpha _1<\cdots <\alpha _k\in\Phi^+,
$
then each factor
$
x_{\alpha _i}(t_i), \ i=1,\dots ,k,
$
lies in~$M$.
\end{lem}

The next statement follows from the definition of a carpet subgroup and Lemma~\ref{lem1}.

\begin{lem}\label{lem2}
Suppose that a subgroup $M\le U(F)$ is normalized by $T(K)$ for a subfield $K\subseteq F$ such that $|K|>4$.
Then the subgroup of $M$ generated by the intersections
$$
M\cap x_\alpha (F)=x_\alpha (\AAA_\alpha ),\ \alpha \in \Phi,
$$
is a carpet subgroup defined by a closed carpet $\AAA=\{\AAA_\alpha \mid \alpha \in \Phi\}$.
\end{lem}

It is well-known that the special linear group of degree two $SL_2(F)$ over a field $F$ is generated by 
the elementary transvections
$$
t_{12}(u)=\begin{pmatrix}1&u\\0&1\end{pmatrix},\ t_{21}(u)=\begin{pmatrix}1&0\\u&1\end{pmatrix},\ u\in F,
$$
and there exists a homomorphism $\varphi$ of $SL_2(F)$ onto the subgroup $\langle X_\alpha ,X_{-\alpha }\rangle$, $\alpha \in\Phi$
that extends the map $t_{12}(u)\to x_\alpha (u),\ t_{21}(u)\to x_{-\alpha }(u)$.

\begin{lem}[{\cite[Lemma~1]{Nuzhin}}]\label{lem3}
Let $r$ be an element of a field $F$ that is algebraic over a subfield $K\subseteq F$ and does not belong to $K$. 
Put $M=\langle t_{21}(K),\ t_{12}(rK)\rangle$. Then one of the following holds{\rm:}
\begin{enumerate}
\item
$|K|=2$ and $M$ is the dihedral group{\rm;}
\item
$|K|=3$, $r^2=-1$, and the image of $M$ in $\opn{PSL}_2(K)$ is isomorphic to $A_5${\rm;}
\item
$M\cap t_{21}(F)\ne t_{21}(K)$.
\end{enumerate}
\end{lem}

Let $K$ be a finite field of characteristic $p$. For $p>2$ Lemma~\ref{lem3} is a particular case of well-known 
theorem of L.~Dickson (see~\cite{DicksonBook} or \cite[теорема 2.8.4]{GorFG}). For $p=2$ it is a particular case
of the main theorem of~\cite{LevchukDickson} by V.~M.~Levchuk, which describes up to equality
all periodic subgroups of $SL_2(F)$ over an arbitrary field $F$, 
that have nontrivial intersections with subgroups of upper and lower unitriangular matrices.
If the field $K$ is infinite, E.~L.~Bashkirov obtained the following improvement of Lemma~\ref{lem3}.

\begin{lem}[see~{\cite{Bash96}}]\label{lem4}
Let $r$ be an element of a field $F$ that is algebraic over a subfield $K\subseteq F$ and does not belong to $K$. 
Suppose that if $\opn{char}F=2$, then $r$ is separable over~$K$. Then
$\langle t_{21}(K),\ t_{12}(rK)\rangle=SL_2\bigl(K(r)\bigr)$.
\end{lem}

A modification of the proof of Lemma~\ref{lem4} for the case where ${\chr K\!\ne\!2}$, or $\chr K\!=\!2$ and $K$ 
is a perfect field is presented in appendix by A.~E.~Zalesski to the paper of F.~G.~Timmesfeld~\cite{Timmesfeld}.

Let us say that roots $\alpha ,\beta\in\Phi$ \emph{commute}, if $\alpha +\beta\notin\Phi$.
The following statement is a particular case of Lemma~2 from~\cite{Nuzhin}.

\begin{lem}\label{lem5}
Let $\Delta$ be a nonempty subset of~$\Phi^+$. Then there exist roots $\alpha \in\Phi^+$ and $\beta\in \Delta$ such that
$\alpha $ commutes with all roots from~$\Delta$ and the inner product $(\alpha ,\beta)$ is nonzero.
\end{lem}

Lemmas~\ref{lem1}--\ref{lem3} and~\ref{lem5} from the works~\cite{Nuzhin,Nuzhin13,NuzhinSkew} of the first author
were key steps in the description of intermediate subgroups between groups of Lie type over different fields
in the case where the bigger field is an algebraic extension of the smaller one.
In a sequel we shall use two special analogs of Lemma~\ref{lem5} for types~$B_l$ and~$C_l$.

\begin{lem}\label{lem6}
Let $\Delta$ be a nonempty subset of~$\Phi^+$, where $\Phi=B_l$, $(l\ge 2)$. 
Then there exists a \emph{long} root $\alpha \in\Phi^+$ and a root $\beta\in \Delta$ such that $\alpha $ 
commutes with all roots from~$\Delta$ and $(\alpha ,\beta)\neq 0$.
\end{lem}

\begin{proof}
The root $\alpha$ from $\Phi^+$ of maximal height is long and commutes with all roots from~$\Phi^+$.
If $\Delta$ contains a root $\beta$ not orthogonal to $\alpha$, then the statement is trivial.
Hence, it suffices to consider the case 
$\Delta\subseteq\Phi^\bot_\alpha \cap\Phi^+$, where $\Phi^\bot_\alpha =\{\gamma\in\Phi\mid (\gamma,\alpha )=0\}$.
For any long root $\alpha$ the set $\Phi^\bot_\alpha$ is an orthogonal sum of a root system of rank~1
consisting of long roots and a root system of type $B_{l-2}$. It is easy to see if the root $\alpha$ is a terminal vertex
in the Coxeter graph. For an orthogonal sum the lemma follows from the statement for irreducible summands.
For one dimensional set $\Delta$ the statement is obvious, therefore, the induction on the rank of $\Phi$
completes the proof.
\end{proof}

\begin{lem}\label{lem7}
Let $\Delta$ be a nonempty subset of~$\Phi^+$, where $\Phi=C_l$ $(l\ge 2)$. Then there exists a short root
$\alpha \in\Phi^+$ and a root $\beta\in\Delta$ such that the root subgroup $X_\alpha(F)$ of a Chevalley group 
$\GG(\Phi,F)$ over a field $F$ of characteristic~$2$ 
commutes with all root subgroups $X_\gamma(F)$, $\gamma\in \Delta$, and $(\alpha ,\beta)\neq 0$.
\end{lem}

\begin{proof}
Let $\alpha$ be the short root from~$\Phi^+$ of maximal height. Then the root subgroup $X_\alpha(F)$
of a Chevalley group $\GG(\Phi,F)$ over a field~$F$ characteristic~$2$ commutes with all root subgroups $X_\gamma(F)$,
$\gamma\in\Phi^+$. 
The rest of the proof is the same as for Lemma~\ref{lem5} with replacing the word \llq long\rrq\
by \llq short\rrq.
\end{proof}

\section{Intermediate subgroups between Chevalley groups of type $B_l$, $C_l$, $F_4$, or $G_2$}\label{sec:Nuzhin}
Theorems~3.1 and~4.1 from~\cite{Nuzhin13} imply the following result.

\begin{prop}\label{prop1}
Let $F$ be an algebraic extension of a field~$K$ of characteristic~$p$ and let $M$ be a group
lying between Chevalley groups $\GG(\Phi,K)$ and $\GG(\Phi,F)$ of type $\Phi=B_l$, $C_l$ $(l\ge 2)$, $F_4$, or $G_2$.
Let $p=2$ if $\Phi =B_l$, $C_l$, or $F_4$, and $p=3$ if $\Phi =G_2$. Then $M$ is a product of 
the carpet subgroup
$\E(\Phi,\AAA)$ and a diagonal subgroup $H_M$ normalizing $\E(\Phi,\AAA)$.
The carpet $\AAA=\{\AAA_\alpha \mid \alpha \in \Phi\}$ is closed and is defined by the equality
$$
\AAA_\alpha =
\begin{cases}
P,& \text{ if } \alpha  \text{ is short}, \\
Q,& \text{ if } \alpha  \text{ is long}.
\end{cases}
$$
for some additive subgroups $P$ and $Q$ of the field~$F$ such that
\begin{equation}\label{inc1}
K\subseteq P^p\subseteq Q\subseteq P\subseteq F.
\end{equation}
Moreover, depending on the type of the Chevalley group $\GG(\Phi,F)$ 
the following additional conditions for~$P$, $Q$ and $H_M$ hold.
\begin{itemize}
\item[(а)]
if $\Phi=B_l$ and $l\ge 3$, then $Q$ is a field{\rm;}
\item[(b)]
if $\Phi=C_l$ and $l\ge 3$, then $P$ is a field{\rm;}
\item[(c)]
if $\Phi=F_4$ or $G_2$, then both additive subgroups $P$ and~$Q$ are fields and $H_M$ is the unit subgroup.
\end{itemize}
\end{prop}
For groups of types~$B_l$ and~$C_l$, where $l\ge3$, this proposition improves Theorem~\ref{main}
of the current article in the case of fields.
The paper~\cite{Nuzhin13} asserts that over fields of characterstic~2 Chevalley groups of types~$B_l$ and~$C_l$ are
isomorphic and, therefore, it considers only type~$B_l$. But actually, as we have shown in section~\ref{sec:exceptional},
the exceptional morphism exists but is an isomorphism only over perfect fields of characteristic~2.
However, the proof for type~$B_l$ is valid also for type~$C_l$.
The only difference is that for~$B_l$ $(l\ge 3)$ the smallest additive subgroup~$Q$ is a field, whereas for~$C_l$ $(l\ge 3)$ 
the largest subgroup~$P$ is a field.

In section~\ref{examples} below we give a negative answer on the following question from~\cite[стр. 160]{Nuzhin13}.

\begin{problem}\label{problem1}
In case of $\Phi=B_l$, $C_l$, are both additive subgroups~$P$ and~$Q$ from Proposition~{\rm\ref{prop1}} fields\/{\rm?}
\end{problem}

Inclusions~\eqref{inc1} follows easily from carpet conditions~\eqref{carpet} applied for the family~$\AAA$ 
from proposition~\ref{prop1}.
However, the inverse implication holds only for types~$F_4$ and~$G_2$. Indeed, for $F_4$ in our situation there are
two nontrivial carpet conditions: $PQ\subseteq P$ and $P^2Q\subseteq Q$, which follows from~\eqref{inc1} 
because~$P$ and~$Q$ are fields. By the same reason all nontrivial carpet conditions $PQ\subseteq P$,
$P^2Q\subseteq P$, $P^3Q\subseteq Q$, $P^3Q^2\subseteq Q$, $QQ\subseteq Q$, and
$2PP\subseteq P$ for a family $\AAA$ of type~$G_2$ follows from~\eqref{inc1} as well. 
In section~\ref{examples} we give examples of pairs of additive subgroups in a nonperfect field
of characteristic~2 satisfying conditions~\eqref{inc1} that do not define carpets
of types~$B_l$ and~$C_l$. Namely, we give a negative answer to the following question.

\begin{problem}\label{problem2}
Let $F$ be an algebraic extension of a nonperfect field~$K$ of characteristic~$2$ and let $P$ and~$Q$ be
additive subgroups of~$F$. Suppose that both $P$ and $Q$ are $K$-modules and that one of them is a field. 
Is it true that the inclusions $PQ\subseteq P$ and $P^2Q\subseteq Q$ follow from the inclusions
$K\subseteq P^2\subseteq Q\subseteq P\subseteq F${\rm?}
\end{problem}

For an additive subgroup~$A$ of a field~$F$ define
$$
A^{-1}=\{0\}\cup \{r^{-1} \mid 0\ne r\in A\}.
$$
In a sequel we need the following statement, which is established in the proof of the main theorem of~\cite[p.~535]{Nuzhin}.
It follows from the fact that $F$ is algebraic over $K$ and the carpet conditions, e.\,g. for type~$B_2$
it follows from the inclusions $PQ\subseteq P$ and $P^2Q\subseteq Q$.

\begin{lem}[см.~\cite{Nuzhin}]\label{lem8}
Let $p$, $K$, $F$, $P$, $Q$ and $\Phi$ are the same as in Proposition~\ref{prop1}.
Then, given $r\in P\sm\{0\}$ $($respectively $r\in Q\sm\{0\})$
the ring $K[r]$ $($respectively $K^p[r])$ is contained in $P$ $($respectively in $Q)$. In particular, $P=P^{-1}$ and
$Q=Q^{-1}$.
\end{lem}

The intermediate subgroups of Chevalley groups of skew types over nonperfect fields of exceptional characteristic
are described in~\cite{NuzhinSkew}.

\section{Examples concerning admissible pairs}\label{examples}
At the beginning of this section we construct counterexamples to Problems~\ref{problem1} and~\ref{problem2}.
For it suffices to define fields  $K\subseteq F$ in the following way.
Let $n$ be a positive integer greater than one and $x_1,\dots ,x_n$ independent commutative variables,
i.\,e. transcendental elements over the field of two elements~$\mathbb{F}_2$. Consider the field of rational functions
$F=\mathbb{F}_2(x_1,\dots ,x_n)$ and its subfield $K=\mathbb{F}_2(x_1^2,\dots ,x_n^2)$,
generated by the squares of the variables. Obviously, $F$ is an algebraic extension of~$K$ of degree~$2^n$ and
the set of monomials
$$
\bigl\{x_{i_1}\dots x_{i_m}\mid i_1<\cdots <i_m,\ \{i_1,\cdots ,i_m\}\subseteq \{1,\dots ,n\}\bigr\}
$$
is a basis of $F$ over $K$.

\begin{prop}\label{counterex1}
Let $\Phi=B_l$, $C_l$, $l\ge2$.
There exist fields $K\subseteq F$ of characteristic $2$ and an admissible pair
$(P,Q)$ of type~$\Phi$ such that $K\subseteq Q\subseteq P\subseteq F$ and:
\begin{enumerate}
\item
if $\Phi=B_l$, $l\ge 3$, then $P$ is not a field;
\item
if $\Phi=C_l$, $l\ge 3$, then $Q$ is not a field;
\item
if $\Phi=B_2=C_2$, then neither $P$, nor $Q$ is a fields.
\end{enumerate}
\end{prop}

\begin{proof}
{\bf Type $B_l$, $l\ge 3$}. Let $Q=K$ and let $P$ be the $K$-module with the basis $\{1,x_1,x_2\}$. 
The following properties can be verified immediately:
(1) $PQ\subseteq P$; (2) $P^2Q\subseteq Q$; (3) $\dim_KP=3$.
Since $3$ does not divide $2^n$, the additive subgroup $P$ is not a field.

\smallskip
{\bf Type $C_l$, $l\ge 3$}. Let $P=F$, and let $Q$ be a $K$-module with the basis $\{1,x_1,x_2\}$. Properties (1)-(3) 
from the previous part of the proof hold also in this case, except that in the third property one substitutes~$P$ 
by~$Q$. The rest of the proof repeated one for type~$B_l$.

{\bf Type $B_2=C_2$}. Suppose in this case that $n\ge 4$. Define the additive subgroups~$P$ and~$Q$ of the field $F$ 
in the following way.
Let $P$ be a $K$-module with the basis
$$
\{1,\,x_1,\,x_2,\,x_3,\,x_4,\,x_1x_2,\,x_1x_3,\,x_1x_4,\,x_2x_3,\,x_2x_4,\,x_1x_2x_3,\,x_1x_2x_4\},
$$
and let $Q$ be a $K$-module with the basis $\{1,x_1,x_2\}$. 
Again, the following properties can be verified immediately:
(1) $PQ\subseteq P$;
(2) $P^2Q\subseteq Q$;
(3) $\dim_KP=12$, $\dim_KQ=3$.
Since $3$ and $12$ do not divide $2^n$, both $P$ and $Q$ are not fields.
\end{proof}

The carpet, corresponding to the admissible pair from the previous proposition gives
the negative answer to Problem~\ref{problem1}.
The following statement gives the negative answer to Problem~\ref{problem2}.

\begin{prop}\label{counterex2}
There exist fields $K\!\subseteq \!F$ of characteristic~$2$ and $K$\nobreakdash-submodules $P$ and $Q$ in $F$ such that
$K\subseteq P^2\subseteq Q\subseteq P\subseteq F$, the module $Q$ {\rm(}respectively $P${\rm)} is a field, 
but the inclusion $PQ\subseteq P$ {\rm(}respectively $P^2Q\subseteq Q${\rm)} does not hold,
i.\,e. the pair $(P,Q)$ is not an admissible pair of type~$B_l$ {\rm(}respectively $C_l${\rm)}.
\end{prop}

\begin{proof}
Define additive subgroups~$P$ and~$Q$ in the field~$F$ as $K$-modules with the bases $\{1,x_1,x_2,\}$ and
$\{1,x_1\}$ respectively. Note that $Q$ is a field. Subgroups~$P$ and~$Q$ satisfy conditions of the proposition,
however $x_1x_2\notin P$, and hence, $PQ\notin P$.

To prove the statement in the case where $P$ is a field we reduce the field~$K$ to
$K=\mathbb{F}_2(x_1^4,\dots ,x_n^4)$. Define additive subgroups~$P$ and~$Q$ of the field~$F$ as $K$-modules
in the following way. Put
$P=F$ and $Q=P^2+Kx_1+Kx_2$. Evidently, $x_1^3\in P^2Q$. However, $x_1^3\notin Q$. Indeed, if ${x_1^3\in Q}$, then
$x_1^3=r+s x_1+t x_2$, where $r\in P^2$, $s,t\in K$.
It follows that
$
0=r+(s-x_1^2)x_1+t x_2
$
and all the coefficients $r$, $s-x_1^2$, and $t$ belong to the field $P^2$. 
Elements $1,x_1,x_2$ are linearly independent over~$P^2$. Therefore, $s-x_1^2=0$,
but $x_1^2\notin K$. The contradiction shows that $P^2Q\notin Q$.
\end{proof}

In~\cite{NuzhYak} the authors describe intermediate subgroups between Chevalley groups $\GG(\Phi,R)$ and $\GG(\Phi,F)$ 
of an arbitrary type $\Phi$ over the fraction field $F$ of the principle ideal domain $R$. 
It turns out that in this case such a subgroup coincides with $\GG(\Phi,P)$ for an intermediate subring~$P$, 
$R\subseteq P\subseteq F$. Why in this case admissible pairs do not appear?
The following proposition answers this question.

\begin{prop}\label{prop2}
Let $F$ be the fraction field of a principle ideal domain $R$ and $Q\subseteq P$ a pair of additive subgroups
in~$F$, containing~$R$. If $P$ is an $R$-module and $P^mQ\subseteq Q$ for some positive integer~$m$, then $P=Q$.

For an admissible pair $(P,Q)$ of an arbitrary type such that $R\subseteq Q$,
we have $P=Q$ and $P$ is a subring of~$F$.
\end{prop}

\begin{proof}
Let $r,s\in R$ and $\frac{r}{s}\in P$. Without loss of generality we may assume that $r$ and $s$ are mutually prime.
Then there exist $u,v\in R$ such that $ur+vs=1$. Since $P$ is an $R$-module and contains $R$, we have
$\frac{1}{s}=u\frac{r}{s}+v\in P$. The product $rs^{m-1}$ lies in $R\subseteq Q$ and in $P^mQ\subseteq Q$, 
therefore,
$\frac{r}{s}=(\frac{1}{s})^m(rs^{m-1})\in Q$. Thus, $P\subseteq Q$, and hence, $P=Q$.

The second assertion follows from the first one and the definition of admissible pair given in
the introduction.
\end{proof}

For Chevalley groups of skew types ${}^2\!A_{2n+1},{}^2\!D_n,{}^2\!E_6,{}^3\!D_4$ over the fraction field~$F$ 
of a PID~$R$ the intermediate subgroups were described in~\cite{MoisNuzhin} with certain restrictions
on the cardinality of the multiplicative subgroup of the ring~$R$.

\section{Bruhat decomposition}\label{sec:Bruhat}
The factorization $G(F)=U(F)N(F)U(F)$ of a Chevalley group $G(F)$ over a field $F$ is usually called the \textit{Bruhat decomposition}. 
Such presentation of an element from $G(F)$ is not unique. 
However, it can be transformed to the \textit{reduced} Bruhat decomposition (canonical presentation of elements of 
Chevalley groups, see~\cite[Corollary~8.4.2]{CarterBook}), which is unique. 
The following theorem establishes the reduced Bruhat decomposition in the carpet subgroup corresponding to an admissible pair.

The image of~$N(\Z)$ in $N(F)$ is denoted by~$N^\pm(F)$ or simply by~$N^\pm$. 
The expressions $T(\AAA)$ and $N(\AAA)$ are used to denote the intersection of the carpet subgroup $\E(\Phi,\AAA)$
with the subgroups $T(F)$ and $N(F)$ respectively. In the unipotent subgroups
$$
U(F)=\langle X_\alpha(F)\mid \alpha\in\Phi^+\rangle\text{ and}
$$
$$
V(F)=\langle X_\alpha(F)\mid \alpha\in-\Phi^+\rangle
$$
a carpet $\AAA$ naturally defines the unipotent carpet subgroups
$$
U(\AAA)=\langle x_\alpha (\AAA_\alpha )\mid \alpha \in\Phi^+ \rangle\text{ and}
$$
$$
V(\AAA)=\langle x_\alpha (\AAA_\alpha )\mid \alpha \in -\Phi^+ \rangle,
$$
These subgroups coincide with the intersections of $\E(\Phi,\AAA)$ with $U(F)$ and $V(F)$ respectively
if and only if the carpet~$\AAA$ is closed. For an element $w\in W(\Phi)$ and a carpet $\AAA$ of type~$\Phi$ 
put
$$
U_w(\AAA)=\langle x_\alpha (\AAA_\alpha )\mid \alpha \in\Phi_w^+ \rangle, \text{ where } 
\Phi_w^+=\Phi^+\cap w^{-1}(-\Phi^+).
$$

\begin{thm}\label{Bruhat}
\sly
Let $\E(\Phi,\AAA)$ be a carpet subgroup lying between Chevalley groups $\GG(\Phi,K)$ and
$\GG(\Phi,F)$ of type $\Phi=B_l,C_l,F_4,G_2$ ($l\ge 2$), where $F$ is an algebraic extension
of a nonperfect field~$K$ of characteristic $p=2$ for $\Phi =B_l,\ C_l,\ F_4$, and $p=3$ for $\Phi =G_2$.
For each element $w\in W$ chose its representative $n_w$ in $N^\pm$.
Then an element $g\in\E(\Phi,\AAA)$ has a unique presentation $g=uhn_wv$, where
$u\in U(\AAA)$, $h\in T(\AAA)$, $w\in W$, and $v\in U_w(\AAA)$. In particular,
the carpet $\AAA$ is closed, i.\,e. the carpet subgroup $\E(\Phi,\AAA)$ 
contains no new root elements.
\end{thm}

\begin{proof}
According to Proposition~\ref{prop1} the carpet~$\AAA$ is defined by a pair of additive subgroups $P$ and $Q$ of the field
$F$, containing~$K$.

Let $g\in\E(\Phi,\AAA)$. Then $g=uhvn$ for some $u\in U(F)$, $v\in V(F)$, $h\in T(F)$,
$n\in N^\pm$, and $n^{-1}vn\in U(F)$ by the reduced Bruhat decomposition in the group $\GG(\Phi,F)$.
Since $N^\pm\le\E(\Phi,\AAA)$, we have $uhv\in\E(\Phi,\AAA)$.

Each element in $u\in U(F)$ can be written in the form
$$
u'a_{\alpha _k}\dots a_{\alpha _1},\qquad \alpha _k>\cdots >\alpha _1,\qquad
a_{\alpha _i}\in X_{\alpha _i}(F)\sm x_{\alpha _i}(\AAA_{\alpha _i}),\
i=1,\dots ,k,
$$
where $u'\in U(\AAA)$, $k\ge 0$, if $k=0$, then $u=u'$, and the ordering of roots is compatible with their heights.
This can be easily deduced from the carpet conditions for the carpet subgroup $U(\AAA)$
see also Lemma~3 from~\cite{NuzhinLevi}.
Similarly, each element $v\in V(F)$ can be written in the form
$$
a_{\beta_1}\dots a_{\beta_m}v',\ \beta_1>\cdots >\beta_m,\ a_{\beta_i}\in X_{\beta_i}(F)\sm x_{\beta_i}(\AAA_{\beta_i}),\
i=1,2,\dots ,m,
$$
where $v'\in V(\AAA)$, $n^{-1}v'n\in U(\AAA)$, $m\ge 0$, and if $m=0$, then $v=v'$.
Since $u',v'\in \E(\Phi,\AAA)$, then $\E(\Phi,\AAA)$ contains the element
$$
a_{\alpha _k}\dots a_{\alpha _1}ha_{\beta_1}\dots a_{\beta_m},\ \alpha _k>\cdots >\alpha _1>0>\beta_1>\cdots >\beta_m.
$$
Suppose that $k+m>0$. By Lemmas~\ref{lem5}-\ref{lem7} there exists a root $\alpha \in-\Phi^+$ such that
the root subgroup $X_\alpha(F)$ commutes with root subgroups
$X_{\beta_i}(F)$ for all $i=1,\dots ,m$ and with $X_{-\alpha _i}$ for all  $i=1,\dots ,k$.
Moreover, $\alpha$ is not orthogonal to one of the roots
$\beta_i$ or $-\alpha _i$. Suppose that it is not orthogonal to $-\alpha _j$, $1\le j\le k$,
if it is orthogonal to all roots $\alpha_i$ but not orthogonal to one of $\beta_i$'s
is essentially the same.
Taking the smallest possible $j$ we may assume that $(\alpha ,\alpha _i)=0$ for all $i<j$.
This implies that $X_{\alpha}(F)$ commutes with $X_{\alpha_i}(F)$, $i=1,\dots ,j$, as in characteristic~2
any two root subgroups corresponding to orthogonal roots commute.

Let $r=\alpha(h)$, so that $hx_\alpha (t)h^{-1}=x_\alpha (rt)$. Put
$$
y=a_{\alpha_k}\dots a_{\alpha_1}\qquad\text{и}\qquad z=a_{\beta_1}\dots a_{\beta_m}.
$$
Then $yhz\in \E(\Phi,\AAA)$ and
$$
(yhz)x_\alpha(\AAA_\alpha)(yhz)^{-1}=yx_\alpha(r\AAA_\alpha)y^{-1}\in \E(\Phi,\AAA).
$$
On the other hand the choice of~$\alpha$ implies that
$$
x_{-\alpha}(\AAA_{-\alpha})=x_{-\alpha}(\AAA_\alpha)=yx_{-\alpha}(\AAA_\alpha)y^{-1}.
$$
Thus,
\begin{equation}\label{inE}
y\cdot\langle x_{\alpha}(r\AAA_\alpha),x_{-\alpha}(\AAA_\alpha)\rangle\cdot y^{-1}\subseteq\E(\Phi,\AAA).
\end{equation}

Now the proof splits in two cases: (1) $\Phi\ne B_2=C_2$; 
(2) $\Phi=B_2=C_2$.

(1) Let $\Phi\ne B_2=C_2$. Here it becomes important that by Lemmas~\ref{lem5}--\ref{lem7} 
we can take long root~$\alpha$ for $\Phi = B_l$ and short root $\alpha$ for $\Phi = C_l$. 
Therefore, according to items (a)--(c) of Proposition~\ref{prop1} the additive subgroup $\AAA_\alpha$ is a field.
If $r\notin\AAA_\alpha$, then by Lemma~\ref{lem3} there exists $s\notin\AAA_\alpha=\AAA_{-\alpha}$ such that
$x_{-\alpha}(s)=yx_{-\alpha}(s)y^{-1}\in\E(\Phi,\AAA)$, which contradicts to the fact that the carpet
$\AAA$ is closed. Therefore,
$r\in\AAA_{-\alpha}=\AAA_\alpha$. Put $y'=a_{\alpha_k}\dots a_{\alpha_{j+1}}$ and
$a_{\alpha_j}=x_{\alpha_j}(q)$.
Note that $[a_{\alpha_j},h_\alpha(t)]=x_{\alpha_j}\bigl(q(t^m-1)\bigr)$ for all $t\in K\sm\{0\}$,
where $m=-2(\alpha_j,\alpha)/(\alpha,\alpha)\ne 0$. Since the field $K$ is not perfect, it is infinite.
Hence, one  can choose $t\in K\sm\{0\}$ such that $t^m-1\ne0$.
Inclusion~\eqref{inE} implies that
$$
[y,h_\alpha(t)]=[y',h_\alpha(t)]^{a_{\alpha_j}}\cdot[a_{\alpha_j},h_\alpha(t)]=
\prod_{\gamma>\alpha_j}x_\gamma(s_\gamma)x_{\alpha_j}\bigl(q(t^m-1)\bigr)\in\E(\Phi,\AAA)
$$
for some $s_\gamma\in F$.
Applying Lemma~\ref{lem1} to this element and the subgroup $\E(\Phi,\AAA)$ we obtain the inclusion 
$q(t^m-1)\in\AAA_{\alpha_j}$, hence,
$q\in\AAA_{\alpha_j}$. This means that the root element $a_{\alpha_j}$ lies in $\E(\Phi,\AAA)$, which is
a contradiction.

(2) Let $\Phi=B_2=C_2$. In this case in view of an example from section~\ref{examples}
both subgroups~$P$ and~$Q$ are not necessarily fields.
According to Lemmas~\ref{lem6} and~\ref{lem7} we can take long or short root $\alpha$.
Let $\alpha=-2\gamma-\delta$ in the notation of item~(1), where $\{\gamma,\delta\}$ is the set of fundamental 
positive roots in~$B_2$. Then
\begin{equation}\label{b21}
(yhz)x_{-2\gamma-\delta}(Q)(yhz)^{-1}=yx_{-2\gamma-\delta}(r Q)y^{-1},
\end{equation}
where $hx_{-2\gamma-\delta}(1)h^{-1}=x_{-2\gamma-\delta}(r)$. On the other hand,
since the root subgroups $X_{2\gamma+\delta}(F)$ and $X_{\gamma+\delta}(F)$ are central in $U(F)$, we have
\begin{equation}\label{b22}
x_{2\gamma+\delta}(Q)=yx_{2\gamma+\delta}(Q)y^{-1},
\end{equation}
\begin{equation}\label{b23}
x_{\gamma+\delta}(Q)=yx_{\gamma+\delta}(Q)y^{-1}.
\end{equation}
Commuting the elements of the shapes~\eqref{b21} and~\eqref{b23}, we obtain
\begin{equation}\label{b24}
yx_{-\gamma}(r)x_\delta(r)y^{-1}.
\end{equation}
Now commuting the elements of the shapes~\eqref{b22} and~\eqref{b24}, we get
\begin{equation}\label{b25}
yx_{\gamma+\delta}(r)x_\delta(r^2)y^{-1}.
\end{equation}
By Lemma~\ref{lem1} equation~\eqref{b25} implies the inclusion $x_\delta(r^2)\in\E(\Phi,\AAA)$. Therefore,
if $r^2\notin Q$, then we arrive at the contradiction with the equation $\E(\Phi,\AAA)\cap X_\delta=x_\delta(Q)$.

Let $r^2\in Q$ but $r\notin Q$. By Lemma~\ref{lem8} $Q^{-1}=Q$. Therefore, the last inclusion is equivalent
to the equation $qr^2=1$ for some nonzero $q\in Q$. Moreover, $q\neq 1$, as $r=1$ when $q=1$ and in this
case we again obtain the inclusion $a_{\alpha_j}\in \E(\Phi,\AAA)$, which leads to a contradiction.
Further, for all $q_i\in Q$ the subgroup $\langle x_{2\gamma+\delta}(Q),\ x_{-2\gamma-\delta}(r Q)\rangle$ 
contains the image of the matrix
\begin{align*}
a&=
\begin{pmatrix}
1&q_4r \\
0&1 \end{pmatrix}
\begin{pmatrix}
1&0 \\
q_3&1 \end{pmatrix}
\begin{pmatrix}
1&q_2r \\
0&1 \end{pmatrix}
\begin{pmatrix}
1&0 \\
q_1&1 \end{pmatrix}
\\
&=
\begin{pmatrix}
1+(q_1q_2+(q_1+q_3)q_4)r+q_1q_2q_3q_4r^2&* \\
(q_1+q_3)+q_1q_2q_3r&*
\end{pmatrix}
\end{align*}
under the homomorphism~$\ph$ from $SL_2(F)$ onto $\langle X_{2\gamma+\delta},X_{-2\gamma-\delta}\rangle$,
that extends the map $t_{12}(u)\to x_{-2\gamma-\delta}(u)$, $t_{21}(u)\to x_{2\gamma+\delta}(u)$. Let
$q_1+q_3=1$, $q_2=q_1^{-1}$, $q_4=1$. Then
$$
a=
\begin{pmatrix}
1+(q_1+1)r^2&* \\
1+(q_1+1)r&* \end{pmatrix}.
$$
Obviously, $(q_1+1)\neq 0,1$ when $q_1\neq 0,1$. Since $qr^2=1$, for $(q_1+1)=q$ we have
$$
a=
\begin{pmatrix}
0&(1+qr)^{-1} \\
1+qr&* \end{pmatrix}
=
\begin{pmatrix}
0&(1+qr)^{-1} \\
1+qr&0 \end{pmatrix}
\begin{pmatrix}
1&* \\
0&1 \end{pmatrix}.
$$
Let $\varphi(a)=z$. Then
\begin{equation}\label{b26}
\begin{split}
yzx_{-2\gamma-\delta}(r Q)z^{-1}y^{-1}&=yx_{2\gamma+\delta}((1+qr)^2r Q)y^{-1}
\\
&=x_{2\gamma+\delta}((1+qr)^2r Q)\subseteq \E(\Phi,\AAA).
\end{split}
\end{equation}
Recall that $qr^2=1$, hence
$$
(1+qr)^2r=(1+q^2r^2)r=(1+q)r=(1+r^{-2})r=r+r^{-1}=
\frac{r^2+1}{r}=\frac{q^{-1}+1}{r}.
$$
Since $0,1\neq q\in Q$, then $0\neq q^{-1}+1\in Q$ according to the equality $Q=Q^{-1}$. 
Therefore formula~(\ref{b26}) implies that $r^{-1}\in Q$, and hence $r\in Q$, which is a contradiction.
\end{proof}

\begin{rem}
In the proof of~\cite[Theorem~3.1]{Nuzhin13} for type~$B_l$
the inclusions of the factors of the product $uhvn\in M$ into~$M$
are claimed to have the same proof as in~\cite{Nuzhin}. 
However, in~\cite{Nuzhin} all additive subgroups $\AAA_\alpha$, which are defined by 
$M\cap x_\alpha(F)=x_\alpha(\AAA_\alpha),\ \alpha\in \Phi$, coincide with a subfield of the ground field~$F$.
The proof of Theorem~\ref{Bruhat} given above at the same time fills this gap 
(when at least one of additive subgroups $\AAA_\alpha$ is not a field).
\end{rem}

\begin{rem}
The Bruhat decomposition of a Chevalley group $G(F)$ over a field~$F$ implies the \textit{Gauss decomposition} 
$G(F)=U(F)T(F)V(F)U(F)$, which holds even for semilocal commutative rings.
In 1976 Z.~I.~Borewich~\cite{BorevichParabolic} established the Gauss decomposition for matrix subgroups of $GL_n$ and $SL_n$, 
that are defined by net of ideals of a semilocal ring. The same result was obtained by N.~A.~Vavilov and 
E.~B.~Plotkin~\cite{VavPlotNet1,VavPlotNet2} for net subgroups of all Chevalley groups over commutative semilocal rings.
\end{rem}

\section{Intermediate subgroups as groups with a $(B,N)$-pair}\label{sec:BN}
Subgroups $B$ and $N$ of an arbitrary group $G$ are called a \emph{$(B,N)$-pair} they satisfy the following axioms.

\begin{itemize}
\item[BN1.]
Subgroups $B$ and $N$ generate $G$.
\item[BN2.]
$B\cap N \unlhd N$.
\item[BN3.]
The quotient group $W=N/B\cap N$ is generated by involutions $w_i$, $i \in I$.
\item[BN4.]
For any preimage $n_i \in N$ of $w_i$ under the natural homomorphism $N$ onto $W$ we have
$$
Bn_iB \cdot BnB \subseteq Bn_inB \cup BnB, \quad n \in N.
$$
\item[BN5.]
If $n_i$ is the element from axiom~$(BN4)$, then $n_iBn_i\neq B$.
\end{itemize}

In different termilogy with $S=\{w_i\mid i\in I\}$ the quadruple $(G,B,N,S)$ is called a \emph{Tits system}
\cite[p. 26]{Bourbaki4-6}).

A $(B,N)$-pair is called \emph{split} if $B=U(B\cap N)$, where
$U$ is a normal nilpotent sbugroup of~$B$, see~\cite[p.~149]{GorBook}.
A $(B,N)$\nobreakdash-pair is called \emph{saturated}, if
$\bigcap_{n\in N}B^n=B\cap N$, see~\cite[p. 58]{Bourbaki4-6}).

It is well-known that a Chevalley group~$\GG(\Phi,F)$ over a field~$F$ admits a split saturated $(B,N)$-pair.
The group $B(F)=U(F)T(F)$ can be taken as $B$, and $N(F)$ as $N$. Note that $\bigl(B(F),N(K)\bigr)$
also is a $(B,N)$-pair of the group $\GG(\Phi,F)$ for an arbitrary subfield $K$ of $F$, 
but it is saturated only if $K=F$.

\begin{thm}\label{BN}
\sly
Let $\E(\Phi,\AAA)$ be a carpet subgroup that lies between Chevalley groups $\GG(\Phi,K)$ and $\GG(\Phi,F)$ 
of type $\Phi=B_l$, $C_l$ $(l\ge 2)$, $F_4$, or $G_2$, where $F$ is an algebraic extension of a nonperfect field~$K$ 
of characteristic $p=2$ for $\Phi=B_l$, $C_l$, $F_4$, and $p=3$ for $\Phi =G_2$.
Then the group $\E(\Phi,\AAA)$ is simple and admits a split saturated $(B,N)$-pair.
\end{thm}

\begin{proof}
By Proposition~\ref{prop1} the subgroup $\E(\Phi,\AAA)$ is parametrized by two additive subgroups~$P$ and~$Q$,
satisfying the conditions $K\!\subseteq \!P^p\!\subseteq\! Q\!\subseteq\! P\!\subseteq \!F$.
Let
$Y(P,Q)\!=\!Y(F)\cap\E(\Phi,\AAA)$, if $Y$ is $U$, $V$, $B$, $N$ or~$H$.

We show that $\bigl(B(P,Q),N(P,Q)\bigr)$ is a required $(B,N)$-pair. The monomial subgroup $N(K)$ 
by definition lies in $\E(\Phi,\AAA)$ and acts by conjugation transitively of its root subgroups $x_\alpha(\AAA_\alpha)$,
$\alpha\in\Phi$. Therefore axioms~$(BN1)$ and~$BN5)$ are satisfied. 
Axioms~$(BN2)$, and $(BN3)$ as well as the facts that the pair $(B(P,Q),N(P,Q))$ is split and saturated
follows easily from the definition of the groups $B(P,Q)$, $N(P,Q)$ and $T(P,Q)$.
The proof of axiom~$(BN4)$ for the whole Chevalley group $\GG(\Phi,F)$ from~\cite[стр. 106]{CarterBook} 
is valid in our situation with obvious changes and, therefore, is omitted. 

Next we establish the simplicity of the group $\E(\Phi,\AAA)$. It is well-known (see e.\,g.~\cite[p. 170]{CarterBook}),
that a group~$G$ with a $(B,N)$-pair is simple if the following conditions hold.

\begin{itemize}
\item[(a)]
$G=G^\prime$,
\item[(b)]
$B$ is solvable,
\item[(c)]
$\bigcap_{g\in G}gBg^{-1}=1$,
\item[(d)]
the set $I$ is not a disjoint union of two nonempty elementwise commuting subsets $J$ and $J'$.
\end{itemize}

The group $\E(\Phi,\AAA)$ is generated by its root subgroups $x_\alpha(\AAA_\alpha)$, $\alpha\in\Phi$ and
$$
[x_\alpha(\AAA_\alpha),T(P,Q)]=x_\alpha(\AAA_\alpha).
$$
Therefore the group $\E(\Phi,\AAA)$ is perfect. Clearly, the group $B(P,Q)$ is solvable. The equality
$$
\bigcap\limits_{g\in \E(\Phi,\AAA)}gB(P,Q)g^{-1}=1
$$
can be established by the same arguments as for the whole Chevalley group $\GG(\Phi,F)$ \cite[p. 172]{CarterBook}. 
Finally, the quotient group $N(P,Q)/T(P,Q)$ is isomorphic to the Weyl group of the system $\Phi$. 
Thus. the $(B,N)$-pair $\bigl(B(P,Q),N(P,Q)\bigr)$ satisfy conditions~(a)-(d), and hence,
the group $\E(\Phi,\AAA)$ is simple.
\end{proof}

The groups $\E(\Phi,\AAA)$ from Theorem~\ref{BN} are interesting also with respect to the following problem
posed by A.~V.~Borovik, see~\cite{KosUnSys}.

\begin{problem}\label{problem3}
Describe infinite groups with a split saturated $(B,N)$-pair.
\end{problem}

\bibliographystyle{amsplain}
\bibliography{english}

\providecommand{\bysame}{\leavevmode\hbox to3em{\hrulefill}\thinspace}
\providecommand{\MR}{\relax\ifhmode\unskip\space\fi MR }
\providecommand{\MRhref}[2]{%
  \href{http://www.ams.org/mathscinet-getitem?mr=#1}{#2}
}
\providecommand{\href}[2]{#2}
\begin{thebibliography}{10}

\bibitem{AbeNormal}
E.~Abe, \emph{Normal subgroups of {C}hevalley groups over commutative rings},
  Contemp. Math. \textbf{83} (1989), 1--17.

\bibitem{AbeSuzuki}
E.~Abe and K.~Suzuki, \emph{On normal subgroups of {C}hevalley groups over
  commutative rings}, Tohoku Math. J. \textbf{28} (1976), no.~2, 185--198.

\bibitem{BakBook}
A.~Bak, \emph{{K}-theory of forms}, Ann. of Math. Stud. \textbf{98}, Princeton
  Univ. Press, Princeton N.J., 1981.

\bibitem{BakStepSymp}
A.~Bak and A.~V. Stepanov, \emph{Subring subgroups of symplectic groups in
  characteristic 2}, St.Petersburg Math. J. \textbf{28} (2017), no.~4,
  465--475.

\bibitem{Bash96}
E.~L. Bashkirov, \emph{On subgroup of the special linear group of degree 2 over
  an infinite field}, Sb. Math. \textbf{187} (1996), no.~2, 175--192.

\bibitem{BorevichParabolic}
Z.~I. Borevich, \emph{Parabolic subgroups in linear groups over a semilocal
  ring}, Vestnik Leningrad Univ., Math., Mech., Astronom. (1976), no.~13,
  16--24.

\bibitem{Bourbaki4-6}
N.~Bourbaki, \emph{Elements of mathematics. {L}ie groups and {L}ie algebras.
  {C}hapters 4-6}, Springer-Verlag, Berlin, 2008.

\bibitem{CarterBook}
R.~Carter, \emph{Simple groups of lie type}, Wiley and Sons, London--New
  York--Sydney--Toronto, 1972.

\bibitem{DemazureGabriel}
M.~Demazure and P~Gabriel, \emph{Introduction to algebraic geometry and
  algebraic groups}, Math. Stud. \textbf{39}, North-Holland, Amsterdam, 1980.

\bibitem{DicksonBook}
L.~E. Dickson, \emph{Linear groups with an exposition of the {G}alois fields
  theory}, Teubner, Leipzig, 1901.

\bibitem{GorFG}
D.~Gorenstein, \emph{Finite groups}, Harper and Row, N.Y., 1968.

\bibitem{GorBook}
\bysame, \emph{Finite simple groups. {A}n introduction to their
  classification}, Plenum Publishing Corp., N.Y., 1982.

\bibitem{Jantzen}
J.~C. Jantzen, \emph{Representations of algebraic groups, 2nd ed.}, Math.
  Surveys Monogr. \textbf{107}, AMS, 2003.

\bibitem{vdKCent}
{W. van der} Kallen, \emph{Another presentation for {S}teinberg groups},
  Nederl. Akad. Wetensch. Proc. Ser. A \textbf{80} (1977), 304--312.

\bibitem{KosUnSys}
A.~I. Kostrikin, {Kh}.~{Ya}. Unachev, and S.~A. Syskin, \emph{Third school on
  finite group theory}, Uspehi matematicheskih nauk \textbf{38} (1983), no.~2,
  236--238.

\bibitem{LavSpCent}
A.~Lavrenov, \emph{Another presentation for symplectic {S}teinberg groups}, J.
  Pure and Appl. Algebra \textbf{219} (2015), 3755--3780.

\bibitem{LavSinDl}
A.~Lavrenov and S.~Sinchuk, \emph{On centrality of even orthogonal
  $\mathrm{K}_2$}, J. Pure and Appl. Algebra \textbf{221} (2017), no.~5,
  1134--1145.

\bibitem{LevchukABA}
V.~M. Levchuk, \emph{Parabolic subgroups of certain {ABA}-groups}, Math. Notes
  \textbf{31} (1982), no.~4, 259--267.

\bibitem{LevchukDickson}
V.~M. Levchuk, \emph{A remark on a theorem of {L}.~{D}ickson}, Algebra Logic
  \textbf{22} (1983), no.~4, 306--316.

\bibitem{MoisNuzhin}
T.~V. Moiseenkova and Ya.~N. Nuzhin, \emph{The {I}wasava decomposition and
  intermediate subgroups of the {S}teinberg groups over the field of fractions
  of a principal ideal ring}, Sci. China Math. \textbf{52} (2009), no.~2,
  318--322.

\bibitem{Nuzhin}
Ya.~N. Nuzhin, \emph{Groups contained between groups of {L}ie type over
  different fields}, Algebra Logic \textbf{22} (1983), 378--389.

\bibitem{Nuzhin13}
\bysame, \emph{Intermediate subgroups in the {C}hevalley groups of type {$B_l$,
  $C_l$, $F_4$}, and {$G_2$} over the nonperfect fields of characteristic $2$
  and $3$}, Sib. Math. J. \textbf{54} (2013), no.~1, 119--123.

\bibitem{NuzhinSkew}
\bysame, \emph{Groups lying between {S}teinberg groups over nonperfect fields
  of characteristics 2 and 3}, Proc. Steklov. Inst. Math. \textbf{285} (2014),
  no.~Suppl.1, 146--152.

\bibitem{NuzhinLevi}
\bysame, \emph{Levi decomposition for carpet subgroups of {C}hevalley groups
  over a field}, Algebra Logic \textbf{55} (2016), no.~5, 367--375.

\bibitem{NuzhYak}
Ya.~N. Nuzhin and A.~V. Yakushevich, \emph{Intermediate subgroups of
  {C}hevalley groups over a field of fractions of a ring of principal ideals},
  Algebra Logic \textbf{39} (2000), no.~3, 199--206.

\bibitem{SinEl}
S.~Sinchuk, \emph{On centrality of $\mathrm{K}_2$ for {C}hevalley groups of
  type ${E}_l$}, J. Pure and Appl. Algebra \textbf{220} (2016), no.~2,
  857--875.

\bibitem{SmolSuzRee}
A.~Smolensky, \emph{Suzuki--{R}ee groups and tits mixed groups over rings},
  Comm. Algebra (2019), DOI:10.1080/00927872.2019.1662913.

\bibitem{SteinChevalley}
M.~R. Stein, \emph{Generators, relations, and coverings of {C}hevalley groups
  over commutative rings}, Amer. J. Math. \textbf{93} (1971), 965--1004.

\bibitem{Steinberg}
R.~G. Steinberg, \emph{Lectures on {C}hevalley groups}, Dept. of Mathematics,
  Yale University, 1968.

\bibitem{StepPolynormal}
A.~V. Stepanov, \emph{On the distribution of subgroups normalized by a fixed
  subgroup}, J. Soviet Math. \textbf{64} (1993), no.~1, 769--776.

\bibitem{StepNormal}
\bysame, \emph{On the normal structure of the general linear group over a
  ring}, J. Math. Sci. (N. Y.) \textbf{95} (1999), no.~2, 2146--2155.

\bibitem{StepGL}
\bysame, \emph{Nonstandard subgroups between {$\mathrm{E}_n(R)$} and
  {$\mathrm{GL}_n(A)$}}, Algebra Colloq \textbf{10} (2004), no.~3, 321--334.

\bibitem{StepNonstandard}
\bysame, \emph{Free product subgroups between {C}hevalley groups
  {$\mathrm{G}(\Phi,F)$} and {$\mathrm{G}(\Phi,F[t])$}}, J. Algebra
  \textbf{324} (2010), no.~7, 1549--1557.

\bibitem{StepStandard}
\bysame, \emph{Subring subgroups in {C}hevalley groups with doubly laced root
  systems}, J. Algebra \textbf{362} (2012), 12--29.

\bibitem{StepGrTh}
\bysame, \emph{Sandwich classification theorem}, Int. J. Group Theory
  \textbf{4} (2015), no.~3, 7--12.

\bibitem{Suzuki}
K.~Suzuki, \emph{On parabolic subgroups of chevalley groups over local rings},
  Tohoku Math. J. \textbf{28} (1976), no.~1, 57--66.

\bibitem{Taddei}
G.~Taddei, \emph{Normalit{\'e} des groupes {\'el\'e}mentaire dans les groupes
  de {C}hevalley sur un anneau}, Contemp. Math. \textbf{55} (1986), 693--710.

\bibitem{Timmesfeld}
F.~G. Timmesfeld, \emph{Abstract root subgroups and quadratic action}, Adv.
  Math. \textbf{142} (1999), 1--150.

\bibitem{VavPlotNet1}
N.~A. Vavilov and E.~B. Plotkin, \emph{Net subgroups of {C}hevalley groups.
  {I}}, J. Soviet Math. \textbf{19} (1982), 1000--1006.

\bibitem{VavPlotNet2}
\bysame, \emph{Net subgroups of {C}hevalley groups. {II}}, J. Soviet Math.
  \textbf{27} (1984), 2874--2885.

\end{thebibliography}
\end{document}